\documentclass[a4paper,11pt]{article} 
  
\usepackage[english]{babel}
\usepackage[a4paper]{geometry}
\usepackage{amsmath}
\usepackage{amsthm}
\usepackage{amssymb}
\usepackage{bbm}
\usepackage{color}
\usepackage{comment}
\usepackage{mathrsfs}

\usepackage{hyperref}         

\numberwithin{equation}{section}

\def\NN{\mathbb{N}}

\def\cD{\mathcal{D}}
\def\cA{\mathcal{A}}

\def\cV{\mathcal{V}}

\def\cF{\mathcal{F}}
\def\cG{\mathcal{G}}
\def\cL{\mathcal{L}}
\def\cJ{\mathcal{J}}
\def\cN{\mathcal{N}}

\def\cH{\mathcal{H}}
\def\cX{\mathcal{X}}

\def\eps{\varepsilon}

\def\calP{\mathcal{P}}
\def\cZ{\mathcal{Z}}

\def\unit{\mathbbm{1}}  

\def\tr{{\rm tr\,}}

\def\bE{{\mathbb{E}}}

\newtheorem{theorem}{Theorem}[section]  

\newtheorem{lemma}[theorem]{Lemma}

                          \let\r=r

\begin{document}

\title{Free Energy of the Quantum Sherrington-Kirkpatrick Spin-Glass Model with Transverse Field}

\author{Arka Adhikari, Christian Brennecke
\\
\\
Department of Mathematics, Harvard University, \\
One Oxford Street, Cambridge MA 02138, USA$$ \\
\\}

\maketitle

\begin{abstract}
We consider the quantum Sherrington-Kirkpatrick (SK) spin-glass model with transverse field and provide a formula for its free energy in the thermodynamic limit, valid for all inverse temperatures $\beta>0$. To characterize the free energy, we use the path integral representation of the partition function and approximate the model by a sequence of finite-dimensional vector-spin glasses with $\mathbb{R}^d$-valued spins. This enables us to use results of Panchenko who generalized in \cite{Pan2,Pan3} the Parisi formula to classical vector-spin glasses. As a consequence, we can express the thermodynamic limit of the free energy of the quantum SK model as the $d\to\infty$ limit of the free energies of the $d$-dimensional approximations of the model. 
\end{abstract} 



\section{Introduction} 

In this note, we consider systems of $N$ spin-$\frac12$ particles, interacting through i.i.d. standard Gaussian couplings $ (g_{ij})_{1\leq i,j\leq N}$. The Hilbert space describing the system is \linebreak $\mathfrak{H}_N= \bigotimes _{i=1}^N\mathbb{C}^{2}$ and the Hamiltonian $H_N:\mathfrak{H}_N\to\mathfrak{H}_N$ reads
		\begin{equation}\label{eq:defHN} 
		H_N = \frac{\beta}{\sqrt N} \sum_{1\leq i, j\leq N} g_{ij} \sigma_{i}^z\sigma_{j}^z  + h\sum_{i=1}^N \sigma_{i}^x.
		\end{equation}
The inverse temperature is denoted by $ \beta >0$ and the transverse field strength is denoted by $ h > 0$. By $ \sigma^\alpha$, $\alpha\in \{x,y,z\} $, we denote the standard Pauli matrices. In the canonical basis of $\mathbb{C}^2$, they are represented by
		\begin{equation*}\label{eq:defPauli}
		\sigma^x= \left(\begin{matrix} 0 & 1\\ 1&0 \end{matrix}\right), \hspace{0.5cm} \sigma^y= \left(\begin{matrix} 0 & -i\\ i&0 \end{matrix}\right),\hspace{0.5cm}\sigma^z= \left(\begin{matrix} 1 & 0\\ 0&-1 \end{matrix}\right).
		\end{equation*}
For $ i\in \{1,\dots, N\}$, the operator $\sigma^\alpha_i $ acts on the $i$-th particle as $ \sigma^\alpha$ and as the identity operator on the remaining particles. The partition function $Z_N$ and the  specific free energy $f_N$ of the system are defined by
		\begin{equation*}\label{eq:deffreeenergy}Z_N = Z_N(\beta,h)= \tr e^{H_N}, \hspace{0.5cm}  f_N= f_N(\beta,h) = \frac1{N}\log Z_N.
		\end{equation*} 
Gibbs expectations will be denoted by $  \langle\cdot \rangle =\langle\cdot\rangle_{\beta,h}  = Z_N^{-1} \tr (\cdot)e^{H_N}$.

By setting $h=0$ in \eqref{eq:defHN}, $H_N$ is diagonal in the standard basis of $\mathfrak{H}_N=\bigotimes _{i=1}^N\mathbb{C}^{2}$ and we recover the classical SK model \cite{SK}. In the thermodynamic limit $N\to\infty$, the free energy of this model is given by the minimum of the Parisi functional, predicted by Parisi in \cite{Par1,Par2}. At high temperature, the system is in the so called replica symmetric phase and the formula for the free energy is quite explicit (in fact, for zero external field the limit $f(\beta)$ of the free energy is $f(\beta)= \beta^2/2$ as long as $\sqrt 2\beta<1$). This has been established rigorously, among other properties of the SK model, by Aizenman-Lebowitz-Ruelle in \cite{AizLeRue}. At low temperature, the system undergoes a phase transition into the replica symmetry breaking phase and understanding the free energy becomes more delicate. Based on a crucial result of Guerra \cite{Gue}, the validity of the Parisi formula for all temperatures was proved by Talagrand in \cite{Tal} (in fact, for more general, even $p$-spin interaction models) and by Panchenko in \cite{Pan}, based on his work \cite{Pan4} (for all mixed $p$-spin interaction models). For a thorough discussion of the classical SK and related spin-glass models, we refer the reader to the standard works \cite{MPV, Pan1,Tal1,Tal2}.  

In the quantum setting, there are relatively few results available compared to the classical setting, both in the physics, let us mention here only \cite{BM, IY, SO, Y2} and refer to the review article \cite{Y1} for further references, as well as the mathematics literature \cite{Cr, LeRoRuSp, MaWar1, MaWar2}. Most important for this note is the work of Crawford \cite{Cr} that deals with the existence and universality of the free energy in the thermodynamic limit. \cite{Cr} proves in particular that the thermodynamic limit of the free energy exists for all $\beta>0$ and is equal to the limit of the quenched free energy, the average of the free energy over the Gaussian disorder of the system. More precisely, it follows from \cite[Theorem 2]{Cr} that
		\[\lim_{N\to\infty}\bE\big| f_N- \bE f_N\big|^3 = 0. \]
Here and in the following, $\mathbb{E}$ denotes the expectation w.r.t. the Gaussian disorder. 

Although \cite{Cr} proves the existence of the specific free energy in the thermodynamic limit for all temperatures $\beta>0$, it does not provide an explicit formula for it. The recent results by Leschke-Rothlauf-Ruder-Spitzer \cite{LeRoRuSp} and by Manai-Warzel \cite{MaWar1, MaWar2}, on the other hand, deal with the free energy of the quantum SK model with transverse field at high temperature and, respectively, the phase diagram of a quantum generalization of the classical REM model \cite{Der1, Der2}. More precisely, \cite{LeRoRuSp} provides an analysis of the so called annealed free energy $ f_N^{\text{ann}}=N^{-1}\log ( \mathbb{E} Z_N )$ and proves, among other properties, that at sufficiently high temperature it is equal to the quenched free energy $\bE f_N$ in the thermodynamic limit. This is analogous to the classical setting \cite{AizLeRue}, where the equality holds true if and only if $\sqrt 2\beta<1$. The quantum behaviour in \cite{MaWar1, MaWar2} is also modeled, similarly as in $H_N$ in \eqref{eq:defHN}, through a transverse field. Moreover, by the results of \cite{MaWar2}, the quantum REM model can be viewed as the $p\to \infty$ limit of the $p$-spin interaction quantum model with transverse field. In contrast to that, the quantum SK model with $H_N$ as defined in \eqref{eq:defHN} corresponds to the $p$-spin interaction model with $p=2$ and is quite different from the quantum REM model. 

While \cite{MaWar1,MaWar2} deal with the quantum REM model and \cite{LeRoRuSp} with a high temperature region of the quantum SK model, the goal of this paper is to provide a formula for the specific free energy of the quantum SK model in the thermodynamic limit, valid for all inverse temperatures $\beta>0$. To achieve this goal, we first use the path integral representation of the partition function $Z_N$ to interpret the model as a classical vector-spin model, similarly as in \cite{Cr, LeRoRuSp}. The free energy of classical vector-spin models has been analyzed by Panchenko who generalized the Parisi formula in \cite{Pan2,Pan3} to the Potts spin glass and more general mixed $p$-spin interaction models with vector-spins that are distributed according to a compactly supported measure in $\mathbb{R}^d$, for any fixed dimension $d\in \mathbb{N}$. Through its path-integral representation, the quantum SK model corresponds to a vector-spin model with the spins taking values in the space of $ \{-1,1\}$-valued c\`adl\`ag-paths and Panchenko's results are not directly applicable. The goal of this work is therefore to approximate the free energy of the quantum model by that of a sequence of finite $d$-dimensional models whose limiting free energies can be determined using the results of \cite{Pan2,Pan3}. The errors of our approximation are uniform in the particle number $N$ and vanish in the limit $d\to\infty$. The thermodynamic limit of the free energy of the quantum SK model can therefore be described as the $d\to\infty$ limit of the limiting free energies of the finite-dimensional approximations. 

\section{Path Integral Formulation and Main Result}

To state our main result, we will first translate the quantum model to a classical vector-spin model. We follow here \cite[Section 4]{Cr} and \cite[Remark A.2 (ii) \& Appendix B]{LeRoRuSp}. We denote by $ \Omega$ the space of $ \{-1,1\}$-valued c\`adl\`ag-paths on the unit interval $ [0;1]\subset \mathbb{R}$. The set $\Omega$ can be turned into a separable, complete metric space and in the following we equip it with the induced Borel $\sigma$-algebra that turns it into a measurable space. We refer to \cite[Chapter 3, Section 12]{Bil1} for the details and for basic properties of $\Omega$.

In view of the path integral representation of $Z_N$, let $ \eta_i$, $i\in \{1,\ldots,N\}$, denote $N$ independent Poisson point processes (see, for instance, \cite[Chapter 3]{LP}) on the measurable space $ \big([0;\infty), \mathcal{B}([0;\infty))\big)$ with intensity measure $h \lambda_1$, where $\lambda_1$ denotes the Lebesgue measure on $[0;\infty)$. For definiteness, let $\eta$ be a Poisson point process defined on some probability space $(\cX, \cA, \nu)$ and let $ \eta_i$ correspond to the $i$-th independent copy of $\eta$ in the product space $ \mathcal{Y} = \bigotimes_{i=1}^N\cX$ with product measure $ \mu = \bigotimes_{i=1}^N\nu$. On $\mathcal{Y}$, we define the random spin paths $\sigma_i$, $i\in \{1,\dots,N\}$, by 
 		\[ [0;1]\ni t\mapsto \sigma_i(t) = (-1)^{\eta_i([0;t])} \in \Omega. \]
Each $ \sigma_i $ is a random variable with values in the space of $ \{-1,1\}$-valued c\`adl\`ag-paths $\Omega$ and its (random) number of flips $ \eta_i([0;t]) \in \NN$ up to time $t\in[0;1]$ is Poisson distributed with mean $ ht\geq 0$. In the following, we will denote the number of flips of the path $\sigma_i$ up to time $t\in[0;1]$ by $\cN_i(t)=\eta_i([0;t])$, for all $i\in \{1,\dots,N\}$.
 
 Given an initial configuration $ \tau = (\tau_1,\dots,\tau_N)\in \{-1,1\}^N$, we denote by $ \calP_{\tau}^{\otimes N}$ the push-forward measure on $\Omega^N$ that is obtained from $ \mu$ under the map
		\[ \mathcal{Y}\ni \omega=(\omega_1,\dots,\omega_N)\mapsto  \big( \sigma_\tau (t)\big)_{0\leq t\leq 1} (\omega) =\Big( \big(\tau_i(-1)^{\eta_i([0;t])(\omega_i)}\big)_{1\leq i\leq N}  \Big)_{0\leq t\leq 1}\in \Omega^N, \] 
where $ \sigma_\tau \in \Omega^N$ denotes the $N$-particle spin path with (non-random) initial condition $ (\sigma_\tau)_i(0) = \tau_i \sigma_i(0) =\tau_i$, for all $i\in\{1,\dots,N\}$. 
Then, it can be shown that the partition function $ Z_N=\tr e^{H_N}$ takes the form
		\begin{equation}\label{eq:partitionpint0}
		\begin{split} 
		Z_N &= e^{Nh}\sum_{\tau \in \{-1,1\}^N} \int_{\Omega^N} d\calP_\tau^{\otimes N}(\sigma)\unit_{\{\sigma(1)=\sigma(0) \}}\;\exp (\cH_{N}( \sigma)),
		\end{split}
		\end{equation}
where the energies $ \cH_{N}:\Omega^N\to \mathbb{R}$ in \eqref{eq:partitionpint0} are given by
		\begin{equation}\label{eq:defHNquantum}
		\begin{split}
		\cH_{N}(\sigma) &= \frac{\beta}{\sqrt N}\int_0^1 ds\; \sum_{1\leq i, j\leq N}g_{ij}\sigma_i(s)\sigma_j(s) = 	\frac{\beta}{\sqrt N}\sum_{1\leq i, j\leq N}g_{ij} \langle  \sigma_i, \sigma_j\rangle_2,
		\end{split} 
		\end{equation}
and where here and in the following, $ \langle\cdot,\cdot\rangle_2$ denotes the inner product on $ L^2([0;1])$. 

One way to prove the identity \eqref{eq:partitionpint0} is to use the Lie product formula for the exponential of a sum of two finite-dimensional matrices (see, for instance, \cite[Theorem VIII.29]{RS1}). Within this approach, one computes $Z_N$ with the canonical basis of $ \mathfrak{H}_N= \bigotimes_{i=1}^N \mathbb{C}^2$ in which all the $ \sigma_i^z, i\in \{1,\dots,N\},$ are diagonal and notices that $\sigma_i^x$ acts as a spin-flip operator that flips the eigenvector of $\sigma_i^z$ with eigenvalue $ \pm1$ to that with eigenvalue $\mp1$. This is explained in \cite[Section 4]{Cr}, in particular \cite[eq. (4.2)--(4.4)]{Cr} and thereafter. An alternative proof of the identity \eqref{eq:partitionpint0} is explained in \cite[Appendix B]{LeRoRuSp}.

Neglecting constant factors, let us define the probability measure $ \calP^{\otimes N}$ through
		\[\begin{split} \calP^{\otimes N}(\cdot) &= \frac1{C_N}\! \sum_{\tau \in \{-1,1\}^N} \!\!\! \unit_{\{\sigma(1)=\sigma(0) \}} \,\calP_\tau^{\otimes N}(\cdot),\;\;\text{for }C_N =\!\!\!\! \sum_{\tau \in \{-1,1\}^N} \int_{\Omega^N}d\calP_\tau^{\otimes N}(\sigma) \unit_{\{\sigma(1)=\sigma(0) \}}
		 \end{split}\]
and study from now on the partition function $ \cZ_N= (C_Ne^{Nh})^{-1}Z_N$, that is
		\begin{equation}\label{eq:partitionpint}
		 \cZ_N = \int_{\Omega^N} d\calP^{\otimes N}( \sigma)\;\exp (\cH_{N}( \sigma)).
		 \end{equation}
Moreover, since $ \Omega \subset L^2([0;1])$ with continuous embedding, we will identify $ \calP^{\otimes N}$ in the following with its push-forward under the map $ \Omega^N \ni\sigma  \mapsto \sigma\in (L^2([0;1]))^N$. That is, we view $ \calP^{\otimes N}$ as a measure on $(L^2([0;1]))^N$. Since $ \|\sigma_1\|_2=1$ for all $\sigma_1\in\Omega$, we have that
 	$$ \operatorname{supp}\big(\calP^{\otimes N}\big) \subset\big\{\varphi\in L^2([0;1]): \|\varphi\|_2=1\big\}^N.$$ 
	
The Gibbs measure induced by \eqref{eq:partitionpint} will be denoted by $ \cG_N$ and it has density 
		\[d\cG_N(\sigma) = \cZ_N^{-1} e^{\cH_N(\sigma)} d\calP^{\otimes N}(\sigma).\]
Gibbs expectations w.r.t. $\cG_N$ will still be denoted by $\langle\cdot \rangle$. We will also denote Gibbs expectations w.r.t. the product measures $ \cG_N^{\otimes j}$ by $\langle\cdot\rangle$, for $j\in \NN \,\cup\, \infty$. As common in the literature, replicas of spin-configurations will be indicated by a super-index. For example, we write $ \sigma^1$ and $\sigma^2$ for replicas in the product space $ (\Omega^N)^2$ with measure $ \cG_N^{\otimes 2}$.

Having translated the quantum model to a classical vector-spin model, let us now turn to the description of our main result. In the classical SK model, the thermodynamic limit of the free energy 
is given by the infimum of the Parisi functional (see, for instance, \cite[Chapter 3]{Pan1}), with the infimum being taken over the set of cumulative distribution functions $\zeta$ on $[0;1]$. Loosely speaking, the unique minimizer (see \cite{AuCh}) of the functional should be thought of as the asymptotic distribution of the so called overlap, the functional order parameter of the system. If the minimizer, that depends on the inverse temperature $\beta>0$, corresponds to the distribution function of a Dirac $\delta$-measure, the system is in the so called replica symmetric phase (this happens at sufficiently high temperature) and otherwise it is in the so called replica symmetry breaking phase (which happens at low enough temperature). In the works \cite{Pan2, Pan3}, Panchenko generalized the Parisi formula to classical mixed $p$-spin interaction models with vector-spins that are distributed according to a compactly supported measure in $\mathbb{R}^d$, for $d\in \mathbb{N}$ (see also \cite[Section 1.12]{Tal1} for a high-temperature region). To state our main result, we will first need to recall and slightly extend some of the definitions from \cite{Pan2,Pan3} to the infinite dimensional quantum setting. 

First of all, the overlap $ Q: \Omega^2\to \cJ_1$ is an operator-valued map from $(\Omega^N)^2$ to the set of trace class operators $\cJ_1$ on $L^2([0;1])$ and defined through its kernel by
		\begin{equation}\label{eq:defoverlap}
		\big(Q(\sigma^1, \sigma^2)\big)(s,t)  = \frac1N \sum_{i=1}^N \sigma_i^1(s)\sigma_i^2(t)=\frac1N \sum_{i=1}^N |\sigma_i^1\rangle\langle\sigma_i^2|(s,t).
		\end{equation}
Here, we introduced in the last equality the usual bra-ket notation for rank-one projections on $L^2([0;1])$. Let us notice that for all $ \sigma^1, \sigma^2\in \Omega^N$, we have that
		\[\tr \big|Q(\sigma^1,\sigma^2)\big| \leq \frac1N\sum_{i=1}^N\tr \big| |\sigma_i^1\rangle\langle\sigma_i^2|\big| = \frac1N\sum_{i=1}^N\| \sigma_i^1\|_2 \|\sigma_i^2\|_2 = 1.  \]
The importance of the overlap $ Q$ can already be seen from the fact that it determines the covariance of the Gaussian process $ (\cH_N(\sigma))_{\sigma\in\Omega^N}$. More precisely, we have that
		\[ \text{Cov} (\cH_N(\sigma^1), \cH_N(\sigma^2)) = N\beta^2 \| Q(\sigma^1,\sigma^2)\|_{HS}^2, \]		
where here and in the following, $ \langle\cdot, \cdot\rangle_{HS}$ and $\|\cdot\|_{HS}$ denote the Hilbert-Schmidt inner product and, respectively, the Hilbert-Schmidt norm of bounded operators on $L^2([0;1])$. 

In analogy to the classical SK model, it has been proved in \cite{Pan2,Pan3} that the role of the functional order parameter of finite-dimensional vector-spin glasses is played by the asymptotic distribution of the overlap $Q$, namely by the asymptotic distribution \linebreak $ \mathbb{E} \cG_N^{\otimes 2}\{(\sigma^1,\sigma^2)\in (\Omega^N)^2: Q(\sigma^1,\sigma^2) \in \cdot\}$ in the limit $N\to \infty$. Roughly speaking, it follows from \cite{Pan2,Pan3} that one should think of the limiting overlap as taking values in the set of positive semi-definite, symmetric matrices and, as in the classical SK model, it then enters the generalized Parisi functional as a variational parameter through its cumulative distribution function.

A complication in the derivation of the Parisi formula in \cite{Pan2,Pan3}, compared to the derivation in case of the classical SK model, is the fact that for vector-spin glasses the so called self-overlap is in general not constant. The self-overlap $ R: \Omega\to \cJ_1 $ corresponds to the diagonal of the overlap $Q$. That is, for $\sigma\in\Omega^N$, $R(\sigma)$ is defined by 
		\begin{equation*}\label{eq:defselfoverlap}
		R(\sigma) =Q(\sigma, \sigma) = \frac1N \sum_{i=1}^N |\sigma_i\rangle \langle\sigma_i|.
		\end{equation*}
We remark that for every $ \sigma\in\Omega^N$, $R(\sigma)=R(\sigma)^*$ is self-adjoint and that
		\[0\leq R(\sigma) \leq 1\hspace{0.5cm}\text{and}\hspace{0.5cm} \tr R(\sigma) = \frac1N\sum_{i=1}^N \tr |\sigma_i\rangle\langle\sigma_i|=\frac1N\sum_{i=1}^N \|\sigma_i\|_2^2=1. \]

Obviously, the operator kernel of the self-overlap $R(\sigma)$ is in general not constant. This is also true in the finite-dimensional setting of \cite{Pan2,Pan3}, in contrast to the classical SK model. A consequence is that the generalized Parisi formula for vector-spin glasses also contains the self-overlap as a variational parameter. In fact, in \cite{Pan2, Pan3}, the thermodynamic limit of the free energy is found by computing it as the supremum over all free energies with fixed self-overlap. This can be motivated by the Potts spin glass model \cite{Pan2}, for instance, where the free energy is asymptotically indeed equal to the maximum over all free energies with fixed self-overlap by classical Gaussian concentration results.

Let us now make the above heuristic picture more precise and define the generalized Parisi functional in close analogy to the finite dimensional vector-spin case \cite{Pan2,Pan3}. First of all, we will model the self-overlap by elements of the set $\Gamma\subset \cJ_1$, defined by
		\begin{equation}\label{eq:defGamma}\begin{split}
		 \Gamma 
		 =\big\{\rho \in \cJ_1: \rho=\rho^*, \;0\leq  \rho \leq 1, \;\tr \rho \leq 1\big\}.
		 \end{split}
		\end{equation}
Given a fixed self-overlap $\rho \in \Gamma$, the distribution of the overlap is modeled as follows. We denote by $\Pi$ be the set of monotonically increasing, left-continuous paths from $[0;1]$ to $\{ \rho\in \cJ_1: \rho=\rho^*\geq 0 \}$, equipped with the norm topology in $ \cJ_1$. For fixed $\rho\in \Gamma$, we define $ \Pi_\rho$ as the set of paths with endpoint $\rho$, i.e.
		\begin{equation*}\Pi_\rho= \{\pi\in\Pi: \pi(0)=0, \pi(1)=\rho \}.
		\end{equation*}
Monotonicity refers here, of course, to monotonicity in the sense of operators. To model the distribution of the overlap, we approximate the overlap by discrete paths in $\Pi_\rho$ (the Parisi functional turns out to be continuous in the paths $ \pi_\rho$). Hence, let $\pi_\rho \in \Pi_\rho $ be s.t.
		\begin{equation}\label{eq:discpath}\pi_\rho (t) _{| (m_{j-1};m_j]}  = \gamma_j, \hspace{0.5cm} j=0,\dots, r,\end{equation}
with $ m_{-1}=0 \leq m_0 \leq \ldots \leq m_r =1$ and $ 0=\gamma_0\leq \gamma_1\leq \ldots\leq\gamma_r=\rho$. One should think of the positive operators $ \gamma_j$ as the possible values of the asymptotic overlap and that they are obtained with probablity $ m_j-m_{j-1}\geq0$, for $j\in\{0,\dots,r\}$. 

For a discrete path $ \pi_\rho \in \Pi_\rho$ as above, we construct next an independent sequence of Gaussian processes $ X_j = (X_j(s))_{s\in[0;1]}$, $j\in\{0,\dots,r\}$, such that
		\begin{equation}\label{eq:covXgamma}
		\mathbb{E} X_j(s)X_j(t) = 2(\gamma_j-\gamma_{j-1})(s,t).
		\end{equation}
Since $0\leq  \gamma_j-\gamma_{j-1} = (\gamma_j-\gamma_{j-1})^*\in \cJ_1$, the Gaussian processes $ X_j $ can be constructed explicitly using a complete orthonormal eigenbasis of $ \gamma_j-\gamma_{j-1}\geq 0$. In particular, the sample paths of each $ X_j $ are $\mathbb{P}-a.s.$ $ L^2([0;1])$-valued. The fields $X_j$ correspond to the so called cavity fields which play a crucial role in the understanding of the SK model, see for instance \cite[Chapter V]{MPV}, \cite[Chapter 1.6 \& 1.7]{Tal1} and \cite[Section 1.3]{Pan1}. 

Finally, in addition to the self-overlap and the overlap distribution, the generalized Parisi functional still contains a third parameter. The latter corresponds to a Lagrange multiplier which ensures that the spin configurations have fixed self-overlap. We model the multiplier by elements in $ \cL_s $, the set of bounded, self-adjoint operators on $ L^2([0;1])$. 

With these preparations, consider a self-overlap $ \rho\in\Gamma$, a discrete path $\pi_\rho\in \Pi_\rho$ and the processes $X_j$ as above, and let $\lambda\in\cL_s$. We then define the random variable $ Y_r$ by
		\begin{equation}\label{eq:defYr} Y_r = \log \int_{\Omega} d\calP(\sigma_1) \exp\bigg( \beta \sum_{j=1}^r \langle\sigma_1, X_j \rangle_2  + \tr \lambda |\sigma_1\rangle\langle\sigma_1| \bigg)
		\end{equation}
and the random variables $ Y_j$, for $ j=0,\dots, r-1,$ inductively through
		\begin{equation}\label{eq:defYj} Y_j = \frac{1}{m_j} \log  \mathbb{E}_{j+1} e^{m_j Y_{j+1}}. 
		\end{equation}
Here, $ \mathbb{E}_{j+1}$ denotes the expectation w.r.t. the process $ X_{j+1}$ only. Finally, setting 
		\begin{equation}\label{eq:defPhi} \Phi ( \pi_\rho,\lambda ) = Y_0, \end{equation}
which is non-random, the generalized Parisi functional $\cF$ is defined by
		\begin{equation}\label{eq:defparisifunc}
		\cF(\rho, \pi_\rho,\lambda )  =  \Phi (\pi_\rho,\lambda )  + \frac{\beta^2}2\int_0^1dt\; \| \pi_\rho (t)\|_{HS}^2  - \frac{\beta^2}2 \| \rho\|_{HS}^2 - \tr \lambda\rho.
		\end{equation} 
As in the finite-dimensional vector-spin case, the functional $ \cF(\rho, \pi_\rho, \lambda)$ is jointly Lipschitz continuous in $ (\rho, \pi_\rho, \lambda)$ (see Lemma \ref{lm:contdep} below).  For fixed $ \rho\in \Gamma$ and $ \lambda\in \cL_s$, we can thus extend $ \cF$ uniquely to any, not necessarily discrete path $ \pi_\rho\in \Pi_\rho$.

Our main result, stated as Theorem \ref{thm:parisiformula} below, describes the thermodynamic limit of the free energy of the quantum SK model as the $d\to\infty$ limit of the free energies of suitable $d$-dimensional vector-spin models, with $d\in\NN$. The limiting free energies of the $d$-dimensional vector-spin glasses are described through a variational formula that involves the generalized Parisi functional \eqref{eq:defparisifunc} restricted to suitable $d$-dependent subsets of the sets $ \Gamma$ of self-overlaps as well as $\cL_s$ of Lagrange multipliers, introduced above. To state this precisely, denote by $ (e_k)_{k\in\NN}$ the standard Fourier basis of $L^2([0,1])$, given by
		\begin{equation}\label{eq:defbasis}e_1(t) = 1, \hspace{0.5cm}e_{2k}(t) = \sqrt 2 \sin(2\pi k t) , \hspace{0.5cm} e_{2k+1}(t)= \sqrt 2 \cos(2\pi k t)\end{equation}
for $t\in [0;1]$ and $k\in \mathbb{N}$. In the following section, we will see that the basis $ (e_k)_{k\in\NN}$ is useful for a finite-dimensional reduction of the quantum model, because it allows us to control the number of flips of the random paths $ \sigma \in \Omega^N$ in a suitable way. Using the basis $ (e_k)_{k\in\NN}$, we define for $d\in\NN$ the sets $\Gamma^{d}$ and $\cL_s^d$ by 
		\begin{equation*} \label{def:drestrictedsets}
		\begin{aligned}
    		& \Gamma^{d}=\text{clos}\bigg(\text{conv}\bigg\{ \sum_{ k,l=1}^{d} \langle e_k, \sigma_1\rangle_2 \langle \sigma_1\,   e_l\rangle_2 | e_k\rangle\langle e_l| \in \cJ_1: \sigma_1\in \Omega \bigg\}\bigg),\\
    		&  \cL_s^d=\big\{A\in \cL_s: \langle e_k, A \,e_l \rangle_2  = 0, \text{ for }\forall \,(k,l) \not \in \{1,\ldots,d\}^2 \big\}.
		\end{aligned}
		\end{equation*}
The closure in the definition of $ \Gamma^{d}$ is taken in $\cJ_1$. Notice also that $ \cL_s^d$ consists of the subset of linear maps in $\cL_s$ that map the vector space $ \cV^d= \operatorname{span}(e_1,\dots,e_d)  $ to itself and vanish on its orthogonal complement in $L^2([0;1])$. Our main result reads as follows.
\begin{theorem}\label{thm:parisiformula}
For any $\beta>0$, the quenched free energy $ N^{-1} \mathbb{E}\log \cZ_N$ satisfies
		\begin{equation}\label{eq:parisiformula}
		\lim_{N\to\infty} \frac1N \mathbb{E}\log \cZ_N = \lim_{d \to \infty} \sup_{\rho\in \Gamma^{d}} \bigg[ \inf_{  \pi_\rho \in \Pi_\rho,\, \lambda \in \cL_s^d } \cF(\rho, \pi_\rho,\lambda )\bigg].
		\end{equation}
\end{theorem}
The next section is devoted to the proof of Theorem \ref{thm:parisiformula}. Using the basis $ (e_k)_{k\in\NN}$, we reduce the computation of the thermodynamic limit of the free energy of the quantum model to the computation of a sequence of free energies of finite dimensional vector-spin models to which we can apply the results of \cite{Pan2,Pan3}. We remark that it would be desirable to write the r.h.s. in \eqref{eq:parisiformula} independently of the dimension $d\in\NN$, by finding the correct $d$-independent domains of the self-overlap $\rho \in \Gamma$ and the Lagrange multiplier $\lambda\in \cL_s$ for the quantum model (notice that $ \Pi_\rho$ depends implicitly also on $d\in\NN$ if the self-overlap lies in $\rho \in \Gamma^d$). This is related to the continuity properties of the map $(\rho, \pi_\rho, \lambda)\mapsto\cF(\rho, \pi_\rho, \lambda) $
and we hope to clarify this point in the future.
\medskip

\noindent\textbf{Acknowledgements.} We thank H. T. Yau for suggesting this problem to us and for many helpful discussions. We thank A. Jagannath and G. Genovese for many helpful discussions and for pointing out several useful references related to this note. Moreover, we thank G. Genovese for several useful comments related to this note.
\section{Finite Dimensional Reduction and Proof of Theorem \ref{thm:parisiformula}}\label{sec:finred}

The goal of this section is to explain the reduction of the quantum model to a suitable sequence of finite dimensional vector-spin models which leads together with the results of \cite{Pan2,Pan3} to a proof of Theorem \ref{thm:parisiformula}. To this end, recall first that we view the path measure $ \calP^{\otimes N}$ as a measure on $(L^2([0;1]))^N$ with support in $\{\varphi\in L^2([0;1]): \|\varphi\|_2=1\}^N$. We will abbreviate in the following $\Sigma^N = (L^2([0;1]))^N $. 

Now, consider the basis $(e_k)_{k\in\NN}$ defined in \eqref{eq:defbasis} and fix a dimension $d\in\NN$. Given a path $\sigma\in \Sigma^N$, we associate to it a new path $ \sigma(d)= (\sigma_1(d),\dots,\sigma_N(d))$ through
		\begin{equation*}\label{eq:defdspins}
		\sigma_i(d) = \sum_{k=1}^d \langle e_k, \sigma_i\rangle_2 e_k \in \cV^d = \operatorname{span}(e_1,\dots, e_d)\subset L^2([0;1])
		 \end{equation*}
and we denote by $ f_N: \Sigma^N \to \Sigma^N$ the map that reduces $\sigma$ to $\sigma(d)$, that is
		\begin{equation}\label{eq:defred}
		\Sigma^N\ni   \sigma \mapsto f_N(\sigma) = \sigma(d)\in \Sigma^N.
		\end{equation}
Loosely speaking, the finite dimensional reduction consists simply of considering the vector-spin glass that is obtained from the quantum model by considering the reduced spins $ \sigma(d)$ instead of the full spins $\sigma\in\Omega^N$. We will use several interpolations to make this precise and to this end, we need to introduce a little further notation.	
		
First, define the dimension reduced overlap $ Q^d: (\Sigma^N)^2 \to \cJ_1$ as well as the dimension reduced Hamiltonian $\cH_N^d: \Sigma^N\to \mathbb{R}$ through
		\begin{equation*}\label{eq:reducedoverlap}
		Q^d(\sigma^1,\sigma^2)= \frac{1}{N}\sum_{i=1}^N|\sigma_i(d)\rangle \langle \sigma_i(d)| = Q( f(\sigma^1),f(\sigma^2)) 
		\end{equation*}
and 
		\begin{equation} \label{eq:redHam}
   		\cH_N^d(\sigma) = \frac{\beta}{\sqrt{N}} \sum_{1\le i, j \le N} g_{ij} \langle \sigma_i(d), \sigma_j(d)\rangle_2 = (\cH_N\circ f_N)(\sigma).
		\end{equation}
In the following, we will need to consider partition functions with a restricted domain of integration. For a subset $ S\subset \Sigma^N$, denote by $\cZ_N(S)$ and $\cZ_N^d(S)$ the partition functions
		\begin{equation}\label{eq:constrainedZ}
		\cZ_N(S) = \int_{S} d\calP^{\otimes N}(\sigma) \exp(\cH_N(\sigma)), \hspace{0.5cm} \cZ_N^d(S) = \int_{S} d\calP^{\otimes N}(\sigma) \exp(\cH_N^d(\sigma)).
		\end{equation}
		
Our first lemma is similar to \cite[Lemma 6]{Cr}. Given $\eps>0$ and setting 
		\begin{equation}\label{eq:defSepsd}
		S_\eps^d = \Big \{ \sigma\in \Omega^N:  \frac1N \sum_{i=1}^N \| \sigma_i- \sigma_i(d)\|_2^2\leq \eps  \Big\},
		\end{equation}
our first goal is to show that the quenched free energy $N^{-1}\mathbb{E}\log \cZ_N$ is asymptotically equal to the restricted free energy $N^{-1}\mathbb{E}\log \cZ_N(S_\eps^d)$
if we choose $\eps\sqrt d$ sufficiently large.
\begin{lemma} \label{lem:Crawford}
Fix $ \eps>0$ and $d\in\NN$ such that $ \eps\sqrt d> 8 (e-1)h+4\beta^2$. Then
		\[\lim_{N\to\infty} \frac{1}{N}\Big(\mathbb{E}\log \cZ_N -\mathbb{E} \log\cZ_N(S_\eps^d)\Big)=0\]
\end{lemma}
\begin{proof}
We proceed in two steps. First, we show that $\lim_{N\to\infty} \mathcal{P}^{\otimes N}((S_\eps^d)^c)=0$ whenever $\eps\sqrt d$ is large enough. To prove this, we compare $ S_\eps^d$ with the set
		$$ S_\eps = \big\{ \sigma\in\Omega^N: \sum_{i=1}^N\cN_i \leq N \eps\sqrt d /4  \big\}\subset \Sigma^N,$$ 
where $ \cN_i $ denotes the number of flips of the path $\sigma_i$ in $[0;1]$, for $i\in\{1,\dots,N\}$. Since the $\cN_i$ are independent Poisson variables with mean $h$ and since the probability measure $ \calP$ is supported in the set of paths with an even number of flips, we find that 
		\[\begin{split}
   		\mathcal{P}^{\otimes N}\big((S_\eps)^c\big) \leq e^{-N\eps\sqrt{d}/2} \int_{\Omega^N} d\mathcal{P}^{\otimes N}(\sigma)\exp\bigg( \sum_{i=1}^N \mathcal{N}_i \bigg)  &
		\leq \bigg( 2\sum_{j=0}^{\infty} \frac{ h^j}{j!} e^{\,j-h} \bigg)^N e^{-N\eps\sqrt{d}/4}\\&\leq e^{2N (e-1)h-N\eps\sqrt{d}/4}\to 0
		\end{split}\]
as $N\to\infty $, whenever $ \eps\sqrt d \geq 8 (e-1)h$. Next, we show that 
		$$ S_\eps = \Big\{ \sigma\in\Omega^N: \sum_{i=1}^N\cN_i \leq  N \eps\sqrt d /4 \, \Big\}\subset \Big \{ \sigma\in \Omega^N:  \frac1N \sum_{i=1}^N \| \sigma_i- \sigma_i(d)\|_2^2\leq \eps  \Big\}=S_\eps^d$$
s.t. $ \lim_{N\to\infty} \mathcal{P}^{\otimes N}((S_\eps^d)^c)=0$ for $ \eps\sqrt d$ large enough. In fact, let $ \sigma = (\sigma_1,\dots,\sigma_N)\in \Omega^N$, let $k\in\mathbb{N} $ with $k\geq 2$ and recall the definition of the basis $(e_k)_{k\in\mathbb{N}}$ in \eqref{eq:defbasis}. We split $ [0;1]$ into the $k$ intervals $I_{jk}=[\frac{j}{k},  \frac{j+1}{k}]$ for $j\in \{0,1,\dots,k-1\}$ and consider two cases. If $ \sigma_i$ does not have a flip in $I_{jk}$, then by periodicity of $ e_{2k}$ and $ e_{2k+1}$, we find that 
		\[ \int_{I_{jk}} dt\;e_{2k}(t) \sigma_i(t)=\int_{I_{jk}} dt\;e_{2k+1}(t) \sigma_i(t)  =0. \]
	Otherwise, if $\sigma_i$ flips in $I_{jk}$, we use $ \|\sigma_i\|_\infty\leq 1, \|e_k\|_\infty\leq \sqrt 2$ and $| I_{jk}| \leq \frac1k $ so that
		\[ \bigg|\int_{I_{jk}} dt\; e_{2k}(t)\sigma_i(t) \bigg| \leq \frac{\sqrt 2}k, \hspace{0.5cm}\bigg|\int_{I_{jk}} dt\; e_{2k+1}(t)\sigma_i(t) \bigg| \leq \frac{\sqrt 2}k.	\]
As a consequence, we obtain that
		\[ | \langle  e_{2k}, \sigma_i\rangle_2|\leq \frac{\sqrt{2}\,\cN_i }k, \hspace{0.5cm} | \langle  e_{2k+1}, \sigma_i\rangle_2|\leq \frac{\sqrt{2}\,\cN_i }k\]
and therefore, by Cauchy-Schwarz and the fact that $ \|\sigma_i\|_2=1$, we have that
		\[\begin{split} 
		\frac1N\sum_{i=1}^N \| \sigma_i-\sigma_i(d)\|_2^2 &= \frac1N\sum_{i=1}^N\sum_{k = d+1}^{\infty}|\langle e_k, \sigma_i\rangle_2|^2 \\
		& \leq  \frac1N\sum_{i=1}^N\sum_{k =\lfloor (d+1)/2\rfloor}^{\infty}\Big( |\langle e_{2k}, \sigma_i\rangle_2|^2 +|\langle e_{2k+1}, \sigma_i\rangle_2|^2\Big)\\
		&\leq \frac{2\sqrt 2}N\sum_{i=1}^N\bigg(\sum_{k = \lfloor (d+1)/2\rfloor}^{\infty} \frac{\cN_i^2}{k^2}\bigg)^{1/2}\leq \frac{4}N\sum_{i=1}^N \frac{\cN_i}{\sqrt d}.
		\end{split}\]
In particular, if $ \sigma\in S_\eps$, then by the previous bound and the definition \eqref{eq:defSepsd}, we see that $ \sigma\in S_\eps^d$. As mentioned already, this implies with the arguments from above that $$ \lim_{N\to\infty} \mathcal{P}^{\otimes N}\big((S_\eps^d)^c\big)=0.$$ 

In the second step, we will estimate the difference of $\frac{1}{N}\mathbb{E}\big(\log \cZ_N - \log\cZ_N(S_\eps^d)\big)$ in terms of $\mathcal{P}^{\otimes N}((S_\eps^d)^c)$. Here, we argue similarly as in \cite[Lemma 6]{Cr}. Observe first that
		\[\begin{split}
		0\leq  \mathbb{E} \Big( \log \cZ_N - \log \cZ_N(S_\eps^d) \Big)  
		\leq \, \log \bigg( 1  + \int_{(S_\eps^d)^c} d\calP^{\otimes N} ( \sigma) \,\mathbb{E}\frac{\exp(\cH_N(\sigma)) }{ \cZ_N(S_\eps^d) }\bigg).
		\end{split}\]
Applying Jensen's inequality w.r.t. the measure $ \big(\calP^{\otimes N} (S_\eps^d)\big)^{-1}d\calP^{\otimes N}(\cdot)_{|S_\eps^d}$ and the fact that $ (\cH_N(\sigma))_{\sigma\in \Omega^N} $ is a Gaussian process with covariance bounded by 
$$ \big| \mathbb{E} \cH_N(\sigma^1)\cH_N(\sigma^2)\big| \leq N\beta^2 \|Q(\sigma^1,\sigma^2)\|_{HS}^2\leq N\beta^2 ,  $$ 
implies that
		\[\begin{split}
		\int_{(S_\eps^d)^c} d\calP^{\otimes N} ( \sigma) \,\mathbb{E}\frac{\exp(\cH_N(\sigma)) }{ \cZ_N(S_\eps^d) }  
		&\leq  \int_{(S_\eps^d)^c} \frac{d\calP^{\otimes N} ( \sigma)}{\calP^{\otimes N} (S_\eps^d)}\mathbb{E} \exp\bigg( \cH_N(\sigma)+  \int_{S_\eps^d} \frac{d\calP^{\otimes N} ( \tau)}{\calP^{\otimes N} (S_\eps^d)}  \cH_N(\tau)   \bigg)\\
		&\leq\frac{\calP^{\otimes N} (S_\eps^d)^c}{2\calP^{\otimes N} (S_\eps^d)}\bigg[ e^{ 2N \beta^2  } + \mathbb{E} \exp\bigg( 2 \int_{S_\eps^d} \frac{d\calP^{\otimes N} ( \tau)}{\calP^{\otimes N} (S_\eps^d)}  \cH_N(\tau)   \bigg)\bigg] \\
		&\leq e^{4N \beta^2  } \frac{\calP^{\otimes N} (S_\eps^d)^c}{\calP^{\otimes N} (S_\eps^d)}=e^{4 N \beta^2  } \frac{\calP^{\otimes N} (S_\eps^d)^c}{1-\calP^{\otimes N} ((S_\eps^d)^c)}
		\end{split}\]
By the first step, if $ \eps\sqrt d > 8(e-1)h+4 \beta^2$ , the previous estimates imply that
		\[0\leq \limsup_{N\to\infty} \frac1N  \Big(\mathbb{E} \log \cZ_N - \mathbb{E}\log \cZ_N(S^d_{\eps}) \Big)\leq \limsup_{N\to\infty}\frac1N \log (1 + e^{-CN } ) =0,\]
where $C=C_{\beta,h,\eps,d}>0$ is some positive constant.	
\end{proof}
We notice that, with the same arguments as in the previous proof, it follows that 
		\begin{equation}\label{eq:lemmacorZNd}
		\lim_{N \to \infty} \frac{1}{N}\Big( \mathbb{E}\log \cZ_N^d - \mathbb{E}\log \cZ_N^d(S^d_\epsilon)\Big) =0.
		\end{equation}
for $\eps\sqrt{d}> 8 (e-1)h+4\beta^2$. Indeed, the only property of the energies $(\cH_N(\sigma))_{\sigma\in\Omega^N}$ that we used in the previous proof is that their covariance is uniformly bounded by $ N\beta^2$. Looking at the definition \eqref{eq:redHam}, this is also the case for $ \cH_N^d$, independently of $d\in\NN$.

The next lemma compares the energies $N^{-1}\mathbb{E}\log\cZ_N^d(S^d_\eps)$ and $N^{-1}\mathbb{E}\log\cZ_N(S^d_\eps)$.
\begin{lemma}\label{lem:FinDimRed}
    Let $S^d_\eps$ be defined as in \eqref{eq:defSepsd}. For any $\eps$ and any $d\in\NN$, we have that
    \begin{equation}
        \frac{1}{N}\big|\mathbb{E} \log \cZ_N(S^d_\eps) -\mathbb{E}\log\cZ^d_N(S^d_\eps)\big| \le 4\beta^2 \sqrt\eps.
    \end{equation}
\end{lemma}
\begin{proof}
The proof follows by interpolation. Consider the Gaussian process
		\begin{equation*}\label{eq:InterpHam1}
    \cH_{N,t}^d(\sigma)= \sqrt{t} \cH_N(\sigma) + \sqrt{1-t} \cH_N^{d}(\sigma),
		\end{equation*}
with $\cH_N(\sigma)$ as in \eqref{eq:defHNquantum} and where $\cH_N^d(\sigma)$ is defined as in \eqref{eq:redHam}, but with Gaussian disorder $ (\widetilde{g}_{ij})_{1\leq i, j\leq N}$ independent of the disorder $ (g_{ij})_{1\leq i,j\leq N}$ in $\cH_N(\sigma)$. We denote by $\cZ_{N,t}^{d}(S_\eps^d)$ the restricted partition function that is defined as in \eqref{eq:constrainedZ}, but with the Hamiltonian $\cH_N$ replaced by the Hamiltonian $\cH_{N,t}^d$. Below, we also denote by $\langle \cdot \rangle_t$ Gibbs expectations corresponding to the Hamiltonian $\cH_{N,t}^d$. 

Differentiating in $t$ and applying Gaussian integration by parts, we find that
\begin{equation} \label{eq:InterpDimRed}
\begin{aligned}
        \partial_t \mathbb{E} \log \cZ_{N,t}^{d}(S^d_\eps) 
        &= \frac{1}{2} \mathbb{E} \Big[\langle \text{Cov}(  
        \cH_N(\sigma^1),  \cH_N(\sigma^1))\rangle_t - \langle\text{Cov}(\cH_N(\sigma^1), \cH_N(\sigma^2))\rangle_t \\
        & \hspace{1 cm} - \langle\text{Cov}( \cH^d_N(\sigma^1), \cH^d_N(\sigma^1)) \rangle_t + \langle \text{Cov}( \cH^d_N(\sigma^1), \cH^d_N(\sigma^2)) \rangle_t \Big].
\end{aligned}
\end{equation}
To estimate the r.h.s. in the last equation, we will use the fact that in $S_\eps^d$, the differences between the covariances of $\cH_N$ and $\cH_N^d$ are small. For $ \sigma^1,\sigma^2\in S_\eps^d$, we bound
\begin{equation*}
\begin{aligned}
    &\big|\text{Cov}(\cH_N(\sigma^1),\cH_N(\sigma^2)) - \text{Cov}(\cH^d_N(\sigma^1), \cH^d_N(\sigma^2))\big|\\
    &\leq N\beta^2  \Big| \|Q(\sigma^1,\sigma^2)\|_{HS}^2- \|Q(\sigma^1(d),\sigma^2(d))\|_{HS}^2 \Big| \\
    &\leq 2\beta^2 \sum_{i=1}^N   \big\| |\sigma_i^1\rangle\langle\sigma_i^2| -   |\sigma_i^1(d)\rangle\langle\sigma_i^2(d)|\big\|_{HS} \\
   &\leq 2\beta^2 \sum_{i=1}^N  \Big( \| \sigma^1-\sigma^1(d)\|_2 + \| \sigma^2-\sigma^2(d)\|_2 \Big) \leq 4 N \beta^2 \sqrt{\eps}.
\end{aligned}    
\end{equation*}
Hence, the r.h.s. in \eqref{eq:InterpDimRed} is bounded by $4 N \beta^2 \sqrt\eps$ and integrating over $t\in [0;1]$ yields 
		$$\frac{1}{N} \big| \mathbb{E} \log \cZ_N(S^d_\eps) -  \mathbb{E} \cZ_N^d(S^d_\eps)\big| \leq 4\beta^2 \sqrt{\eps}.$$
\end{proof}
Now, choose a fixed constant $ C=C_{h,\beta}>0$ s.t. $ C> 8(e-1)h+4\beta^2$ and set $\eps = \frac{C}{\sqrt{d}}$. Then, the two previous Lemmas \ref{lem:Crawford} and \ref{lem:FinDimRed} imply together with equation \eqref{eq:lemmacorZNd} that
\begin{equation}\label{eq:findlimit}
\lim_{N \to \infty} \frac{1}{N} \mathbb{E} \log \cZ_N = \lim_{d \to \infty}\Big(\lim_{N \to \infty} \frac{1}{N} \mathbb{E} \log \cZ^d_N\Big) .
\end{equation}

By \eqref{eq:findlimit}, we have reduced the problem of computing the free energy of the quantum model to computing the free energy of the vector-spin model with Hamiltonian $\cH_N^d$. In the next step, we show that the latter computation can be done with the help of \cite{Pan2,Pan3} that deals with finite-dimensional vector-spin glasses. 
To this end, let $\mathcal{P}_d$ denote the push-forward measure on $ \mathbb{R}^d$ that is obtained from the path measure $ \calP$ under the map 
		\[  L^2([0;1])\ni \sigma_1\mapsto \big(\langle e_1, \sigma_1\rangle_2, \dots, \langle e_d, \sigma_1\rangle_2\big)\in\mathbb{R}^d .\] 
A first important observation is that $ \operatorname{supp}(\calP_d)$ is a compact subset of $\mathbb{R}^d$. In fact, 
		\[ \big\| \big(\langle e_1, \sigma_1\rangle_2, \dots, \langle e_d, \sigma_1\rangle_2\big)\big \|_{\mathbb{R}^d}^2= \sum_{k=1}^d |\langle e_k, \sigma_1\rangle_2|^2 = \|\sigma_1(d)\|_{2}^2\leq \|\sigma_1\|_2^2 \]
for every $\sigma_1\in L^2([0;1])$ and since $ \operatorname{supp}(\calP)\subset \{ \varphi\in L^2([0;1]):\|\varphi\|_2 =1\} $, it follows that 
		$$ \operatorname{supp}(\calP_d)\subset \big\{ v\in\mathbb{R}^d:\|v\|_{\mathbb{R}^d} \leq 1\big\}. $$
Let's denote from now on the compact support of $\calP_d$ by $\operatorname{supp}(\calP_d) =\mathcal{S}_d$. Furthermore, if we denote by $ \iota: \cV^d =\operatorname{span}(e_1,\ldots,e_d)\to \mathbb{R}^d$ the linear isometry that sends
		\begin{equation}\label{eq:defiota}\cV^d \ni\sum_{k=1}^d  v_k e_k \mapsto (v_1, \dots, v_d )\in\mathbb{R}^d, \end{equation}
let's record that $ \calP_d$ is the push-forward obtained from $\calP$ under the map
		\[ L^2([0;1]\ni \sigma_1\mapsto \iota (\sigma_1(d)) \in \cV^d.\]
Similarly, the push-forward obtained from $\calP^{\otimes N} $ under the map $ \iota_N \circ f_N : \Sigma^N\to\mathbb{R}^{dN}$, with $f_N:\Sigma^N\to\Sigma^N$ defined in \eqref{eq:defred} and $ \iota_N: \Sigma_N\to (\mathbb{R}^{d})^N$ acting in each of the $N$ components as $\iota:L^2([0;1])\to \mathbb{R}^d$, is precisely $ \calP_d^{\otimes N}$ supported in $\mathcal{S}_d^N\subset\mathbb{R}^{dN}$. 

Now, define the overlap matrix of two $N$-tuple collections $\mathcal{\sigma}^1=(\mathcal{\sigma}^1_1,\ldots,\mathcal{\sigma}^1_N)$ and $\mathcal{\sigma}^2=(\mathcal{\sigma}^2_1,\ldots,\mathcal{\sigma}^2_N)$  of vectors in $\mathbb{R}^d$ as in \eqref{eq:defoverlap} by
		\begin{equation*} \label{def:PushforwardOverlap}
    \mathcal{Q}^d(\mathcal{\sigma}^1,\mathcal{\sigma}^2):= \frac{1}{N} \sum_{i=1}^N |\mathcal{\sigma}^1_i\rangle \langle \mathcal{\sigma}^2_i|
		\end{equation*}
and let $\mathscr{H}_N^d$ be a Gaussian process on $\mathbb{R}^{dN}$ with covariance 
		\begin{equation*} \label{def:PushforwardHam}
    \text{Cov}(\mathscr{H}_N^d(\sigma^1), \mathscr{H}_N^d(\sigma^2)) =  N\beta^2 \|\mathcal{Q}^d(\sigma^1,\sigma^2) \|^2_{HS}.
		\end{equation*}
Notice that $\mathscr{H}_N^d:\mathbb{R}^{dN}\to \mathbb{R}$ can be constructed explicitly as	
		\[ \mathscr{H}_N^d(\sigma) = \frac{\beta}{\sqrt N} \sum_{1\leq i, j\leq N}g_{ij} \langle\sigma_i, \sigma_j\rangle_{\mathbb{R}^d}\] 
with the Euclidean inner product $\langle\sigma_i, \sigma_j\rangle_{\mathbb{R}^d}$ on $\mathbb{R}^d$ and i.i.d. standard Gaussian couplings $ (g_{ij})_{1\leq i,j\leq N}$. After these preparations, the key observation is that
		\begin{equation}\label{eq:compZNd}
		\begin{split}
   		\frac1N\mathbb{E}\log \cZ_N^d &=  \frac{1}{N}\mathbb{E} \log \int_{\Omega^N}d\calP^{\otimes N}(\sigma) \exp{ (\mathcal{H}_N^d(\sigma))}  \\
		&= \frac{1}{N}\mathbb{E} \log \int_{\Omega^N}d\calP^{\otimes N}(\sigma) \exp{ (\mathscr{H}_N^d\circ \iota_N\circ f_N)(\sigma))}\\
		&= \frac{1}{N} \mathbb{E}\int_{\mathcal{S}_d^N}  d\calP_d^{\otimes N}(v)\exp( \mathscr{H}_N^d(v) ),
		\end{split}
		\end{equation}
which follows from \eqref{eq:redHam} and by the standard change of variables formula for push-forward measures, applied to $ \calP^{\otimes N}$ obtained from $\calP_d^{\otimes N}$ under the map $ \iota_N\circ f_N:\Sigma^N\to \mathbb{R}^d$.

Putting together \eqref{eq:findlimit} and \eqref{eq:compZNd}, $\lim_{N\to\infty}N^{-1} \mathbb{E}\log \cZ_N$ can be computed by computing the $d\to\infty$ limit of the limiting quenched free energies in the last line of \eqref{eq:compZNd}. But the latter are the free energies of classical vector-spin models with $ \mathbb{R}^d$-valued spins, distributed according to the compactly supported measure $\calP_d$ on $\mathbb{R}^d$. Such models have been analyzed in \cite{Pan2,Pan3}. In the notation of \cite{Pan3}, 
the Hamiltonian $\mathscr{H}^d_N$ can also be written as a mixed $p$-spin interaction Hamiltonian with vector-spins as in \cite[eq. (5)]{Pan3} with $\kappa=d$, $\beta_2(k) = \beta$ for $k\le d$ and $\beta_p(k)=0$ for $p \ge 3$. To determine the $N\to\infty$ limit of the free energy in the last line of \eqref{eq:compZNd}, we apply \cite[Theorems 1 \& 2]{Pan3}.
\begin{proof}[Proof of Theorem \ref{thm:parisiformula}.]
Combining \eqref{eq:findlimit}, \eqref{eq:compZNd} and \cite[Theorems 1 \& 2]{Pan3}, the proof of Theorem \ref{thm:parisiformula} is now straight-forward. In the notation of \cite{Pan3}, we first recall \cite[Theorems 1 \& 2]{Pan3} that determine the thermodynamic limit of the quenched free energy in the last line of \eqref{eq:compZNd}. After that, we explain how this implies Theorem \ref{thm:parisiformula}.

Let $ \Pi_{\rho^d}^d$ denote the set of left-continuous, monotone paths from $[0;1]$ to the space of $\mathbb{R}^{d\times d}$ positive semi-definite, symmetric matrices with endpoint $ \rho^d$. As in equation \eqref{eq:discpath}, consider a discrete path $\pi^d_{\rho^d}\in \Pi_{\rho^d}^d$ which is determined by the parameters \linebreak $ m_{-1}=0 \le m_0\le \ldots \le m_{r-1} \le m_{r}=1$
 and $ 0=\gamma^d_0 \le \gamma^d_1\le \ldots \le \gamma^d_{r-1}\le \gamma^d_{r}=\rho^d $ s.t.
 		\begin{equation}\label{eq:ddiscpath} \pi_{\rho^d}^d (t) _{| (m_{j-1};m_j]}  = \gamma_j^d, \hspace{0.5cm} j=0,\dots, r.\end{equation}
 Define i.i.d. Gaussian random vectors $X^d_j= \big(X^d_j (s)\big)_{1\leq s\leq d}$, for $j \in \{1,\ldots,r\}$, s.t.
		\begin{equation*}\label{eq:Xjd} \mathbb{E} X^d_j (s)X_j^d(t) = 2(\gamma_j^d - \gamma_{j-1}^d)(s,t)
		\end{equation*}
for $(s,t)\in \{1,\dots,d\}^2$ and let $\lambda^d=(\lambda^d)^*\in \mathbb{R}^{d\times d}$ be a symmetric matrix. Define $Y^d_r$ by
		\begin{equation*}\label{eq:Yrd}
    		Y^d_r= \log \int_{\mathcal{S}_d}d\calP_d(v) \exp\bigg(\beta \sum_{j=1}^r \langle v,X^d_j \rangle_{\mathbb{R}^d} + \tr_{\mathbb{R}^{d\times d}} \lambda^d |v\rangle\langle v|\bigg) 
		\end{equation*}
and, inductively for $j=0,\ldots r-1$, set 
		\begin{equation*}\label{eq:Yjd}
		Y^d_{j} = \frac{1}{m_j} \log \mathbb{E}_{j+1} e^{ m_{j} Y^d_{j+1}}.
		\end{equation*}
We define $\Phi^d(\pi^d_{\rho^d}, \lambda^d)=Y_0^d$ and the Parisi functional $ \cF^d$ by 
		\begin{equation*} \label{def:finFunctional}
    		\mathcal{F}^d(\rho^d,\pi^d_{\rho^d},\lambda^d) = \Phi^d(\pi^d_{\rho^d},\lambda^d) + \frac{\beta^2}{2} \int_0^1 dt\, \|\pi_{\rho^d}^d(t)\|_{HS}^2-\frac{\beta^2}{2}\|\rho^d\|_{HS}^2 -\tr_{\mathbb{R}^{d\times d}} ( \lambda^d \rho^d),
		\end{equation*}
Finally, with 
		\begin{equation*}
    		\cD^{d}=\overline{\text{conv}\big\{ |\sigma_1\rangle\langle\sigma_1| : \sigma_1 \in \operatorname{supp} (\mathcal{P}_d) \big\}}\subset \mathbb{R}^{d\times d},
		\end{equation*}
the main result of \cite{Pan3} states that (see \cite[Theorems 1 \& 2]{Pan3}) 
\begin{equation}\label{eq:Panresults}
    \lim_{N \to \infty} \frac{1}{N} \mathbb{E} \log \cZ_N^d = \sup_{\rho^d \in \cD^{d}} \inf\Big\{  \mathcal{F}^d(\rho^d,\pi^d_{\rho^d},\lambda^d ): \pi^d_{\rho^d} \in \Pi^d_{\rho^d}, \lambda^d=(\lambda^d)^*\in  \mathbb{R}^{d\times d }  \Big\}.
\end{equation}

To finish the proof of Theorem \ref{thm:parisiformula}, we express the r.h.s. in \eqref{eq:Panresults} through the generalized Parisi functional $\cF$, defined in \eqref{eq:defparisifunc}. To this end, we recall the definitions \eqref{eq:defGamma} to \eqref{eq:defparisifunc} from the previous section and we identify a given self-overlap $ \rho^d \in \cD^d$, discrete path $ \pi_{\rho^d}^d\in \Pi_{\rho^d}^d$, the Gaussian vectors $X^d_j$, $j \in \{1,\ldots,r\}$ and Lagrange multiplier $ \lambda^d\in\mathbb{R}^{d\times d}$ with a corresponding self-overlap $ \rho \in \Gamma^d$, discrete path $ \pi_\rho\in \Pi_\rho $, Gaussian processes $X_j=(X_j(s))_{s\in[0;1]}$ and Lagrange multiplier $ \lambda\in \cL_s^d$ (and vice versa) such that 
		\[\mathcal{F}^d(\rho^d,\pi^d_{\rho^d},\lambda^d ) = \mathcal{F}(\rho,\pi_{\rho},\lambda).\]
Apparently, this will imply Theorem \ref{thm:parisiformula}. First of all, we define $\rho\in \Gamma^d $, $\gamma_j\in \Gamma^d $ for $j \in \{0,\ldots,r\}$ and $ \lambda\in\cL_s^d$, by
		\[\begin{split}
		\rho  &= \sum_{k,l=1}^d  ( \rho^d)_{kl} |e_k\rangle\langle e_l| \in \Gamma^d, \hspace{0.3cm}\gamma_j = \sum_{k,l=1}^d  ( \gamma_j^d)_{kl} |e_k\rangle\langle e_l|\in \Gamma^d, \hspace{0.3cm} \lambda = \sum_{k,l=1}^d  ( \lambda^d)_{kl} |e_k\rangle\langle e_l|\in \cL_s^d.
		\end{split}\]
The discrete path $  \pi_\rho\in\Pi_\rho$ is then constructed analogously to \eqref{eq:ddiscpath} through
  		\[ \label{eq:discpathd} \pi_{\rho} (t) _{| (m_{j-1};m_j]}  = \gamma_j, \hspace{0.5cm} j=0,\dots, r.\]
Since $\cV^d\simeq \mathbb{R}^d$ are isometrically isomorphic, we observe at this point already that
		\[ \frac{\beta^2}{2} \int_0^1 dt\, \|\pi_{\rho^d}^d(t)\|_{HS}^2-\frac{\beta^2}{2}\|\rho^d\|_{HS}^2 -\tr_{\mathbb{R}^{d\times d}}( \lambda^d \rho^d) = \frac{\beta^2}{2} \int_0^1 dt\, \|\pi_{\rho}(t)\|_{HS}^2-\frac{\beta^2}{2}\|\rho\|_{HS}^2 -\tr \lambda \rho.\]
It thus only remains to define appropriate random variables $ Y_j$, $ j\in\{ 0,\dots,r\}$, corresponding to the random variables $ Y_j^d$ in the finite-dimensional setting, as defined above, such that $ Y_0^d=Y_0$. Here, we define the Gaussian processes $X_j=\big(X_j(s)\big)_{s\in[0;1]}$ through
		\[X_j = \sum_{k=1}^d (X_j^d)_k e_k. \]
We then have for all $s,t\in [0;1]$ and all $\sigma_1\in \Omega$ that
		\[  \big\langle \sigma_1, X_j\big\rangle_2  = \big\langle \iota(\sigma_1(d)),X^d_j \big\rangle_{\mathbb{R}^d}, \]
where we recall the defintion of the isometry $\iota:\cV^d\to \mathbb{R}^d$ from \eqref{eq:defiota}. Thus, by the change of variables formula for push-forward measures, we conclude that		
		\[\begin{split}
		Y^d_r &= \log \int_{\mathcal{S}_d}d\calP_d(v) \exp\bigg(\beta \sum_{j=1}^r \langle v,X^d_j \rangle_{\mathbb{R}^d} + \tr_{\mathbb{R}^{d\times d}}( \lambda^d |v\rangle\langle v|)\bigg) \\
		& = \log \int_{\Omega}d\calP(\sigma_1) \exp\bigg(\beta \sum_{j=1}^r \langle \iota(\sigma_1(d)),X^d_j \rangle_{\mathbb{R}^d} + \tr_{\mathbb{R}^{d\times d}} \big(\lambda^d |\iota(\sigma_1(d))\rangle\langle \iota(\sigma_1(d))|\big)\bigg)\\
		& = \log \int_{\Omega}d\calP(\sigma_1) \exp\bigg(\beta \sum_{j=1}^r \langle \iota(\sigma_1(d)),X^d_j \rangle_{\mathbb{R}^d} + \tr \lambda |\sigma_1\rangle\langle \sigma_1|\bigg)\\
		&=\log \int_{\Omega}d\calP(\sigma_1) \exp\bigg(\beta \sum_{j=1}^r \langle \sigma_1,X_j \rangle_2 + \tr \lambda |\sigma_1\rangle\langle \sigma_1|\bigg) = Y_r
		\end{split}\] 
with $Y_r$ defined in \eqref{eq:defYr}. Consequently, we find that $ Y_j^d = Y_j$ with $Y_j$ from \eqref{eq:defYj}, for all $j\in\{0,\dots,r\}$. In particular, $ Y_0^d=Y_0$ then implies that
		\[\mathcal{F}^d(\rho^d,\pi^d_{\rho^d},\lambda^d ) = \mathcal{F}(\rho,\pi_{\rho},\lambda).\]
Combining this with \eqref{eq:findlimit}, \eqref{eq:compZNd} and \eqref{eq:Panresults}, we find that
		\[ \lim_{N \to \infty} \frac{1}{N}\mathbb{E} \log \cZ_N= \lim_{d\to\infty} \bigg[ \lim_{N \to \infty} \frac{1}{N}\mathbb{E} \log \cZ_N^d\bigg] = \lim_{d\to\infty} \sup_{\rho \in  \Gamma^{d}}\bigg[ \inf_{ \pi \in \Pi _\rho, \lambda\in \cL_s^d} \cF(\rho, \pi_{\rho},\lambda)\bigg].\]
This concludes the proof of Theorem \ref{thm:parisiformula}. \end{proof}

\appendix
\section{Auxiliary Results}\label{appx:contresults}

In this appendix, we recall the proof of the Lipschitz continuity of the Parisi functional, defined in \eqref{eq:defparisifunc}. For $\lambda\in \cL_s$, we denote by $ \|\lambda\|$ its operator norm.
\begin{lemma} \label{lm:contdep} There exists a constant $C=C_\beta>0$ such that for any $ \rho_0, \rho_1 \in \Gamma$, \linebreak $ \lambda_0,\lambda_1 \in \cL_s$ and discrete $ \pi_{\rho_0} \in \Pi_{\rho_0},  \pi_{\rho_1}\in \Pi_{\rho_1}$, it holds true that
		\[ | \Phi( \pi_{\rho_0}, \lambda_0) -  \Phi(\pi_{\rho_1}, \lambda_1) | \leq C \Big( \| \lambda_0-\lambda_1\| +  \tr |\rho_0-\rho_1| +  \int_0^1dt\; \tr \big| (\pi_{\rho_0}-\pi_{\rho_1})(t) \big| \Big).\]
\end{lemma}
\begin{proof} The proof follows closely \cite[Lemma 7]{Pan2}. We apply Gaussian interpolation and use a representation of the functional $\Phi$, defined in \eqref{eq:defPhi}, in terms of the Ruelle probability cascades (see \cite[Section 2.3]{Pan1}, \cite[Section 13.1]{Tal2} for their definition and basic properties).

Let $ \rho_0, \rho_1 \in \Gamma$, let $ \lambda_0, \lambda_1\in \cL_s$ and let $ \pi_{\rho_0} \in \Pi_{\rho_0},  \pi_{\rho_1}\in \Pi_{\rho_1}$ be discrete paths, as in \eqref{eq:discpath}. Without loss of generality, we can assume that the paths are defined in terms of the same sequence $ (m_j)_{-1\leq j\leq r}$ with $ m_{-1}=0$ and $m_r=1$. We denote the overlap parameters of $ \pi_{\rho_i}$ (for $i=0,1$) by $0=\gamma_{0,i}\leq \gamma_{1,i}\leq \ldots\leq \gamma_{r,i}=\rho_i$ and we associate to $ (\gamma_{j,i})_{0\leq j\leq r}$ Gaussian processes $ (X_{i,j})_{0\leq j\leq r}$ as in \eqref{eq:covXgamma} s.t. $ X_{i,j}$ has covariance $ \gamma_{i,j}-\gamma_{i,j-1}\geq 0$ (here and in the following, we assume that the processes for $ i=0$ are independent from the processes for $i=1$). Furthermore, we denote by $( \nu_\alpha)_{\alpha\in \mathbb{N}^r }$ the weights of the Ruelle probability cascades corresponding to the sequence $ m_{-1}=0<m_1< \ldots<m_r = 1$ and, indexed by $ \alpha\in \mathbb{N}^r$, we construct (for $i=0,1$) a sequence of Gaussian processes $Z_{\alpha,i} =  \big(  (Z_{\alpha,i}(s))_{s\in[0;1]}\big)_{\alpha\in\mathbb{N}^r} $ with covariance
		\[ \mathbb{E} Z_{ \alpha^1,i} (s) Z_{\alpha^2,i}(t) =  \gamma_{\alpha^1\wedge \alpha^2,i}(s,t), \]
where we set
	\[ \alpha^1\wedge\alpha^2 = \min \{0\leq j\leq r: \alpha_1^1 = \alpha_1^2,\dots, \alpha_j^1=\alpha_j^2, \alpha_{j+1}^1\neq \alpha_{j+1}^2 \}.\]
We remark that, using independent copies of the processes $(X_{j,i}(s))_{s\in[0;1]}$ as above, it is straightforward to construct the Gaussian processes $ (Z_{\alpha,i})_{\alpha\in \mathbb{N}^r}  $ explicitly (see, for instance, \cite[Sections 2.3 and 3.1]{Pan1}). We skip the details to avoid further notation.

A basic property of the weights $( \nu_\alpha)_{\alpha\in \mathbb{N}^r }$ (see \cite[Theorem 2.9]{Pan1}) is that 
		\[ \Phi (\pi_{\rho_i}, \lambda_i) = \mathbb{E} \log \sum_{\alpha\in \mathbb{N}^r} \nu_\alpha\int_{\Omega} d\calP(\sigma_1) \exp\Big( \beta \langle \sigma_1, Z_{\alpha,i}\rangle_2  + \tr \lambda_i |\sigma_1\rangle\langle\sigma_1|  \Big) \]
which enables us to interpolate between $ \Phi ( \pi_{\rho_0}, \lambda_0)$ and $\Phi ( \pi_{\rho_1}, \lambda_1)$. To this end, for $ t\in [0;1] $ and $\alpha\in \mathbb{N}^r $, let's denote by $(Z_{\alpha,t})_{\alpha \in \mathbb{N}^r}$ the family of processes
		\[ Z_{\alpha,t} = \sqrt{t} Z_{\alpha,0} + \sqrt{1-t} Z_{\alpha,1}. \]
The covariance between $Z_{\alpha^1,t} $ and $Z_{\alpha^2,t}  $ is consequently	
		\[ \gamma_{\alpha^1\wedge \alpha^2,t} = t\gamma_{\alpha^1\wedge \alpha^2,0} + (1-t) \gamma_{\alpha^1\wedge \alpha^2,1}\geq 0.\]
To interpolate between $ \Phi ( \pi_{\rho_0}, \lambda_0)$ and $\Phi ( \pi_{\rho_1}, \lambda_1)$, we finally define $ \Xi: [0;1]\to \mathbb{R} $ by 
		\[\begin{split} \Xi(t) &= \mathbb{E} \log \sum_{\alpha\in \mathbb{N}^r} \nu_\alpha\int_{\Omega} d\calP(\sigma_1) \exp\Big ( \beta \langle \sigma_1, Z_{\alpha,t}\rangle_2 +\tr (t\lambda_0 + (1-t)\lambda_1)|\sigma_1\rangle\langle\sigma_1|   \Big).
		\end{split}\]
Looking at the definition of $\Xi$, we notice $ \Xi(0) = \Phi(\pi_{\rho_0}, \lambda_0)$ while $ \Xi(1) = \Phi(\pi_{\rho_1}, \lambda_1)$. Now, let's compute $ \partial_t \Xi$. Applying Gaussian integration by parts, we find that
		\[\begin{split}
		(\partial_t\Xi )(t) 
		= &\; \frac{\beta^2}{2}\Big( \mathbb{E}\big\langle \langle \rho_0-\rho_1, |\sigma_1\rangle\langle\sigma_1|\rangle_{HS} \big\rangle_t   - \mathbb{E} \big\langle \langle \sigma_1^1, (\gamma_{\alpha^1\wedge\alpha^2,0}-\gamma_{\alpha^1\wedge\alpha^2,1})\sigma_1^2\rangle_2 \big\rangle_t\Big)  \\
		&\;+ \mathbb{E}\big\langle \tr (\lambda_0-\lambda_1)|\sigma_1\rangle\langle\sigma_1|  \big\rangle_t,
		\end{split}\]
where we set $ \langle\cdot\rangle_t$ to be the time-dependent Gibbs measure on $ \Omega\times \mathbb{N}^r$ with density
		\[ \nu_\alpha d\calP(\sigma_1) \exp\Big( \beta \langle \sigma_1, Z_{\alpha,t}\rangle_2  + \tr (t\lambda_0+(1-t)\lambda_1)|\sigma_1\rangle\langle\sigma_1| \Big) \]
and where $ \sigma_1^1, \sigma_1^2$ denote as usual two replicas of the system. In particular, we obtain 
		\[\begin{split}
		\sup_{t\in [0;1]}|(\partial_t\Xi )(t) |\leq &\; \| \lambda_0-\lambda_1\| +\frac{\beta^2}2  \tr| \rho_0-\rho_1| + \frac{\beta^2}2 \mathbb{E} \big\langle \tr\big| \gamma_{\alpha^1\wedge\alpha^2,0}-\gamma_{\alpha^1\wedge\alpha^2,1}\big| \big\rangle_t. 
		 \end{split}\]
To evaluate the r.h.s. of the last bound further, we note that it follows from general properties of the Ruelle probability cascades (see \cite[eq. (2.82) \& Theorem 4.4]{Pan1}) that 
		\[\begin{split} \mathbb{E} \big\langle \tr\big| \gamma_{\alpha^1\wedge\alpha^2,0}-\gamma_{\alpha^1\wedge\alpha^2,1}\big| \big\rangle_t =&\; \sum_{\alpha^1,\alpha^2\in \mathbb{N}^r} \nu_{\alpha^1}\nu_{\alpha^2} \tr\big| \gamma_{\alpha^1\wedge\alpha^2,0}-\gamma_{\alpha^1\wedge\alpha^2,1}\big|\\
		= &\; \sum_{0\leq j \leq r }\tr\big|  \gamma_{j,0}-\gamma_{j,1} \big|\sum_{\alpha^1\wedge\alpha^2 = j} \nu_{\alpha^1}\nu_{\alpha^2} \\
		= &\; \sum_{0\leq j \leq r }(m_{j}-m_{j-1})\tr \big |  \gamma_{j,0}-\gamma_{j,1}\big| \\
		=&\; \int_0^1dt\; \tr\big | ( \pi_{\rho_0}-\pi_{\rho_1})(t)\big|.
		\end{split}\]
Thus, summarizing the steps from above, we conclude that
		\[\begin{split} | \Phi(\pi_{\rho_0}, \lambda_0)- \Phi(\pi_{\rho_1}, \lambda_1)| 
		&\leq C\Big( \| \lambda_0-\lambda_1\| +  \tr| \rho_0-\rho_1 |+  \int_0^1dt\; \tr\big| (\pi_{\rho_0}-\pi_{\rho_1})(t)\big|\Big),
		\end{split}\]
where $ C = C_\beta = 1+ \beta^2/2$, as desired.
\end{proof}
It follows as a corollary that for fixed $ \rho\in \Gamma$ and $ \lambda\in \cL_s$, we can extend $ \cF(\rho, \pi_\rho, \lambda)$ to any path $\pi_\rho\in \Pi_\rho$. Notice that $ \cF$ was defined in \eqref{eq:defparisifunc} initially only for discrete paths. To extend $  \cF(\rho, \cdot, \lambda)$, let $ \pi_\rho \in \Pi_\rho$. Since $\pi_\rho $ is monotone, it is clear that also $ \tr \pi_\rho $ is a monotonically increasing function in $[0;1]$. By standard arguments, we can approximate $t\mapsto  \tr  \pi_\rho(t)$ in $ L^1([0;1])$ monotonically from below by a sequence of monotonically increasing simple functions, which, by monotonicity of $\pi_\rho$ itself, can be associated to a monotonically increasing sequence of discrete paths in $\Pi_\rho$. If we denote this monotone sequence by $ (\pi_{\rho,n})_{n\in\NN}$, the previous lemma implies together with the fact that \linebreak $ 0\leq \tr | \pi_{\rho,n} - \pi_{\rho,m}| = \tr ( \pi_{\rho,n} - \pi_{\rho,m}) =|  \tr \pi_{\rho,n} - \tr\pi_{\rho,m}|$ whenever $ n\geq m$, that  
		\[ \lim_{ n,m\to\infty} | \cF(\rho, \pi_{\rho,n}, \lambda) - \cF(\rho, \pi_{\rho,m}, \lambda)| =0.\]
We denote its limit by $ \cF(\rho, \pi_\rho,\lambda)$ and, for fixed $\rho\in\Gamma$, the map $ (\pi_\rho, \lambda)\mapsto \cF(\rho, \pi_\rho, \lambda)$ is jointly Lipschitz continuous w.r.t. the metric 
		\[ \big((\pi_\rho, \lambda), (\pi'_\rho, \lambda') \big)\mapsto \| \lambda-\lambda'\| + \int_0^1dt\; \tr\big| (\pi_\rho-\pi'_\rho)(t)\big|. \]

\end{document}